\documentclass[10pt,reqno,oneside]{amsart}

\usepackage{amsmath, amsfonts, amssymb, latexsym, amsthm}

\usepackage{geometry}
\usepackage{mathrsfs}
\usepackage{graphicx} % add figures
\usepackage{verbatim} % useful for program listings
\usepackage{color}    % color in text
\usepackage{subfigure}% side-by-side figures
\usepackage{hyperref} % hypertext links

\usepackage{fancyhdr} % also for special heading

\newtheorem{theorem}{Theorem}
\theoremstyle{plain}

\newtheorem{lemma}[theorem]{Lemma}
\newtheorem{definition}[theorem]{Definition}
\newtheorem{assumption}[theorem]{Assumption}

\newtheorem{remark}[theorem]{Remark}

\numberwithin{equation}{section}
\numberwithin{theorem}{section}

\newcommand{\cA}{\mathcal{A}}

\newcommand{\cD}{\mathcal{D}}

\newcommand{\cF}{\mathcal{F}}

\newcommand{\cS}{\mathcal{S}}

\newcommand{\E}{\mathbb{E}}
\newcommand{\R}{\mathbb{R}}
\newcommand{\F}{\mathbb{F}}

\newcommand{\bP}{\mathbb{P}} % \P is used by the system

\newcommand{\wt}{\widetilde}
\newcommand{\ds}{\displaystyle}
\newcommand{\Ito}{It\^{o}}
\newcommand{\e}{\varepsilon}

\begin{document}

%%%%%%%%%%%%%%%%%%%%%%%%%%%%%%%%%%%%%%%%%%%%%%%%%%%
\title[Uncertainty principle and unique continuation property for SHE]
{Hardy's Uncertainty principle and unique continuation property for
  stochastic heat equations}
%%%%%%%%%%%%%%%%%%%%%%%%%%%%%%%%%%%%%%%%%%%%%%%%%%%

\author{Aingeru Fern\'{a}ndez-Bertolin}
\author{Jie Zhong}
\keywords{Hardy uncertainty principle, unique continuation, stochastic heat equation}
\subjclass[2010]{35B05, 35B60, 60H15}
\address{A. Fern\'andez-Bertolin: Departamento de Matem\'aticas,  Universidad del Pa\'is Vasco UPV/EHU, apartado 644, 48080, Bilbao, Spain}
\email{aingeru.fernandez@ehu.eus}
\address{J. Zhong: Department of Rochester, University of Rochester, NY 14627, USA}
\email{jiezhongmath@gmail.com}

\maketitle

\begin{abstract}
The goal of this paper is to prove a uniqueness result for a stochastic heat equation with a randomly perturbed potential, which can be considered as a variant of Hardy's uncertainty principle for stochastic heat evolutions.
\end{abstract}

\section{Introduction}
It is well known that the unique continuation property has extensive applications in control theory of partial differential equations, especially observability for the system; see \cite{Zuazua2007} for details, or \cite{YZ16} for a stochastic case.

In this paper, we extend a uniqueness result of deterministic equations to the following stochastic heat equation with
multiplicative noise:
\begin{equation}
  \label{eq:main eqn}
  \begin{cases}
  du = (\Delta u + V(t,x)u)\, dt + G(t,x)u\, dW(t), \quad (t,x)\in (0,1]\times \R^n,\\
 u(0) = u_0,
\end{cases}
\end{equation}
which formally can be viewed as a heat evolution with a randomly perturbed potential $V + G\dot{W}$.

Our goal is to understand sufficient conditions for the solution $u$ of equation \eqref{eq:main eqn}, the potential $V$, the noise $G$ and the behavior of the solution at two different times $t_0=0$ and $t_1=1$, in order to guarantee that $u\equiv0$. In the deterministic case, there is a series of papers \cite{cekpv,ekpv1,Escauriaza2008,ekpv3,ekpv5,ekpv4,ekpv6}, where the authors solve this problem for the Schr\"{o}dinger and heat equations. The methodology involved in their project is very robust, as it can be seen in extensions of their results to the magnetic Schr\"{o}dinger equation \cite{bfgrv,cf} and more recently to the discrete Schr\"{o}dinger equation \cite{fb3,fbv,jlmp}. Here we aim to adapt their methods to the stochastic setting.

The motivation of proving unique continuation properties for solutions of Schr\"{o}dinger or heat equations knowing the behavior of the solution at two different times comes from the very famous result of G. H. Hardy \cite{ha} or \cite[page 131]{ss}, concerning the decay of a function $f$ and its Fourier transform
\[
\hat{f}(\xi)=(2\pi)^{-\frac{n}{2}}\int_{\mathbb{R}^n}e^{-i\xi\cdot x}f(x)\,dx.
\]

Under this definition of the Fourier transform, Hardy proves:

 \textit{If $f(x)=O(e^{-|x|^2/\beta^2}),\ \hat{f}(\xi)=O(e^{-4|\xi|^2/\alpha^2})$ and $\alpha\beta<4$, then $f\equiv0$. Also, if $\alpha\beta=4,\ f$ is a constant multiple of $e^{-|x|^2/\beta^2}$.}

Since its original formulation, the Hardy uncertainty principle has been extended to more general settings. For instance, we have the following $L^2-$version of the uncertainty principle \cite{sst}:
\[
e^{|x|^2/\beta^2}f,\ e^{4|\xi|^2/\alpha^2}\hat{f}\in L^2(\mathbb{R}^n)\text{ and }\alpha\beta\le 4 \Longrightarrow f\equiv0.
\]

Moreover, thanks to the expression of solutions of free Schr\"{o}dinger and heat equations, it is possible to rewrite the Hardy uncertainty principle in terms of solutions of these equations. Since we are concerned with the heat equation, in this case it is known that
\[
f,\ e^{|x|^2/\delta^2}e^{\Delta}f\in L^2(\mathbb{R}^n)\text{ for some }\delta\le 2\Longrightarrow f\equiv0.
\]

Thanks to logarithmic convexity properties of solutions with fast decay properties at two different times, the authors extend in \cite{Escauriaza2008,ekpv6} this dynamic Hardy uncertainty principle to solutions of the equation $\partial_tu=\Delta u+Vu$, where the potential $V$ is bounded, using only real variable techniques, whereas the previous known proofs of the Hardy uncertainty principle, up to the endpoint case, were based on complex analysis arguments. In the preliminary non-sharp version of the result in \cite{Escauriaza2008}, they prove first that a solution with Gaussian decay at time $t_0=0$ and $t_1=1$ preserves this decay at any time in  between, and, furthermore, in the open interval $(0,1)$ the solution exhibits better decay properties. Combining this result with a Carleman estimate, they are able to conclude uniqueness for solutions with a non-sharp rate of decay. Every step of the proof follows a formal approach that is justified at the end of the proof, which represents a considerable technical difficulty.

It is reasonable to think that in the presence of a noise term, the statement will not change, at least for small noises. We see in this paper that the approach introduced in \cite{Escauriaza2008} can be adapted to our setting to extend the Hardy uncertainty principle, but the rate of the decay depends on the noise. However, this result is likely to be improved, but we do not have a hint about the sharp rate of the decay at the moment.

We need to assume the following hypothesis on the potential $V$ and the noise $G$ in equation \eqref{eq:main eqn}.

\begin{assumption}
  \label{assump:V-G}
The measurable functions $V$ and $G: [0,1]\times\R^n \to \R$ satisfy the following conditions:
\begin{itemize}
  \item[(1)] $V$ is bounded on $[0,1]\times\R^n$, and $G\in C^1_b([0,1]\times \R^n);$
  \item[(2)] Given $\gamma>0$, there exists $\e>0$ such that 
\[
\ds \sup_{t\in [0,1]}|G(t,x)|\le \frac{\sqrt{\gamma}}{(\sqrt{1+4\gamma})^{1-\e}}|x|^{-\e},\quad\text{for}~|x|\ge \frac{\max\{\gamma,2\}}{\sqrt{1+4\gamma}};
\]
  \item[(3)] $\displaystyle \lim_{L\to\infty} \sup_{t\in[0,1],|x|>L} |V(t,x)|= \lim_{L\to\infty} \sup_{t\in[0,1],|x|>L} |\nabla G(t,x)|=0$.
\end{itemize}
\end{assumption}

Notice that the potential $V$ considered in the deterministic case is bounded whereas in the stochastic setting we require it to slightly decay at infinity. If we only carry out the formal arguments, we do not need neither the potential nor the noise to decay, but, trying to rigorously prove our statement, just a bounded potential and noise is not enough. This is due to the fact that the procedure to prove the logarithmic convexity result (see Lemma \ref{lem:lem3}) is different in the stochastic case, since the justification process fails unless we know first some decay properties of the solution in the interior of the interval $[0,1]$. Nevertheless, if the noise $G = G(t)$, is a constant or independent of space variable $x$, it is easy to see from some obvious transform that the deterministic result still holds. The later statement can also be verified through the proof of the logarithmic convexity in Lemma \ref{lem:lem3}, for details see Remark \ref{rem:G(t)}. Therefore, in the sequel we only focus on the case that $G$ depends on $x$. 

Before we state our main theorem, let us introduce some basic notations.

Let $\F = (\Omega, \cF,\{\cF_t\}_{t\ge 0}, \bP)$ be a stochastic
basis with usual conditions. On $\F$, we define a standard scalar Wiener
process $W = \{W(t)\}_{t\ge 0}$. We assume that the
filtration 
$\{\cF_t\}_{t\ge0}$ is generated by $W$.

Given a Hilbert space $H$, we denote by $L^2_\cF([0,1];H)$ the Banach
space consisting of all $H$-valued $\{\cF_t\}_{t\ge0}$-adapted
processes $X$ such that the square of the canonical norm  $\E\int_0^1
\|X(t)\|_H^2\, dt
<\infty$; and denote by $C_\cF([0,1];H)$ the Banach
space consisting of all $H$-valued $\{\cF_t\}_{t\ge0}$-adapted continuous
processes $X$ such that the square of the canonical norm $\E
\sup_{0\le t\le 1}
\|X(t)\|^2_H<\infty$. 

We denote by $(\cdot,\cdot)$ the inner product
in $L^2(\R^n)$ and denote by $\|\cdot\|$ the norm induced by $(\cdot,
\cdot)$. We also use the notation $\|f\|_\infty =
\text{ess}\sup_{(t,x)\in [0,1]\times \R^n}|f(t,x)|$.

\begin{definition}
  We say $u$ is a solution of equation \eqref{eq:main eqn} if $u$ is
  in the space of\\ $C_\cF([0,1];L^2(\R^n))\bigcap
  L^2_\cF([0,1];H^1(\R^n))$ such that
for all $\varphi\in C^\infty_c(\R^n)$ and all $t\in [0,1]$ we have
\[
(u(t),\varphi) = (u_0, \varphi) +\int_0^t (\nabla u(s), \nabla
\varphi)\, ds + \int_0^t (Vu(s), \varphi)\, ds + \int_0^t (Gu(s),
\varphi)\, dW(s), \ \bP\text{-a.s.}
\]
\end{definition}

The following is our main result.

\begin{theorem}
  \label{th:main}
Suppose $u$ is a solution of equation \eqref{eq:main eqn}, and assume that
\[
\E \|u_0\|^2<\infty, \quad \E \| e^{|x|^2/\delta^2}u(1)\|^2<\infty,
\]
for some $0<\delta<1$. Then we have $u\equiv 0$ in $[0,1]\times \R^n$, $\bP$-a.s. if either $V$ is bounded in $[0,1]\times \R^n$ and $G = G(t)$ is a bounded function in $[0,1]$, or Assumption \ref{assump:V-G} holds and $\gamma = 1/(2\delta)$ satisfies
\begin{equation}
  \label{eq:gammaG}
  \frac{4\gamma^2-1}{8\alpha_\gamma\gamma(1+4\gamma)}> \|G\|^2_\infty,
\end{equation}
where
\begin{equation}
  \label{eq:alpha-gamma}
  \alpha_\gamma = 
\begin{cases}
  \frac{1}{4}+\frac{1+\sqrt{8\gamma^2+1}}{16\gamma^2}, &\mbox{if}~\gamma \ge \frac{1+\sqrt{2}}{2},\\   \frac{2\gamma+1}{8\gamma^2}+\frac{\sqrt{8\gamma+3}}{16\gamma^2}, &\mbox{if}~\frac{1}{2}<\gamma< \frac{1+\sqrt{2}}{2}.
\end{cases}
\end{equation}
\end{theorem}

As we have pointed out above this result is not likely to be sharp, and a combination of it with the procedure developed in \cite{ekpv6} will probably start a self-improvement argument. On the other hand, it is reasonable to claim that a similar result holds for Schr\"odinger evolutions. We are currently working on both projects.

The rest of the paper is organized as follows: in Section 2, we provide necessary estimates, especially the interior regularity for the decay of the solution. In Section 3, we first introduce a formal calculation leading to the logarithmic convexity, and then focus on rigorous justifications. Section 4 is devoted to the proof of our main result Theorem \ref{th:main}. 

\section{Preliminary estimates}
In this section, we start with the energy estimate for the solution $u$ of equation \eqref{eq:main eqn}, multiplied by a quadratic exponential weight function.
\begin{lemma}[Energy estimate]
  \label{lem:energy}
Suppose $u$ is a solution of equation \eqref{eq:main eqn}. Then there
is a constant $C>0$ such that
\[
\E \sup_{0\le t\le 1}\left\|e^{\phi_\gamma(t)|x|^2}u(t)\right\|^2 \le
e^{CM_{G,V}}\E \|e^{\gamma |x|^2} u_0\|^2,
\]
where $\gamma\ge 0, \phi_\gamma(t) = \gamma/(1+4\gamma t)$, and $M_{G,V} = \|G\|^2_\infty + 2\|V\|_\infty$.
\end{lemma}
\begin{proof}
Formally,  let $v = e^{\varphi(t,x)}u$ with $\varphi(t,x) = \phi_\gamma(t)|x|^2$, then by \Ito's formula we have
\begin{align*}
  dv & = \partial_t \varphi v\, dt + [-\Delta\varphi v - 2\nabla \varphi\cdot
       \nabla v + \Delta v +|\nabla\varphi|^2v + Vv]\, dt +  Gv\, dW(t).
\end{align*}
Applying \Ito's formula for $\|v\|^2$ and integration by parts yield
\begin{multline*}
\|v(t)\|^2 + 2\int_0^t \|\nabla v\|^2\, ds = \|v(0)\|^2 + 2\int_0^t
\int_{\R^n} \left(|\nabla\varphi|^2 + \partial_t\varphi\right) v^2\,
dx ds\\
+\int_0^t \|Gv(s)\|^2\, ds + 2\int_0^t(v, Vv)\, ds + 2\int_0^t
(v,Gv)\, dW(s).
\end{multline*}
It is clear that 
\[
|\nabla\varphi|^2 + \partial_t\varphi = 0,
\]
and thus we obtain 
\begin{multline}
  \label{eq:engery eq}
\|v(t)\|^2 + 2\int_0^t \|\nabla v\|^2\, ds = \|v(0)\|^2\\
+\int_0^t \|Gv(s)\|^2\, ds + 2\int_0^t(v, Vv)\, ds + 2\int_0^t
(Gv,v)\, dW(s).
\end{multline}
Taking expectation on both sides and getting rid of the gradient term,
we have
\begin{align*}
  \E \|v(t)\|^2 & \le \E \|v(0)\|^2 +  \E \int_0^t \|Gv(s)\|^2\, ds +
                  2\E \int_0^t |(v,Vv)|\, ds\\
&\le \E \|v(0)\|^2 + \|G\|^2_\infty\, \E \int_0^t \|v(s)\|^2\, ds +
  2\|V\|_\infty\, \E \int_0^t \|v(s)\|^2\, ds\\
& = \E \|v(0)\|^2 + M_{G,V}\E \int_0^t \|v(s)\|^2\, ds.
\end{align*}
It follows from Gronwall's inequality that
\[
\E \|v(t)\|^2 \le e^{M_{G,V}} \E \|v(0)\|^2,
\]
which also implies that $\E \int_0^1 \|v(t)\|^2\, dt \le e^{M_{G,V}} \E
\|v(0)\|^2$. 

Now we go back to the equality \eqref{eq:engery eq}, and use the
Burkholder-Davis-Gundy's inequality with $p=1$ to estimate $\E
\sup_{0\le t\le 1}$ as follows:
\begin{align*}
  \E \sup_{0\le t\le 1} \|v(t)\|^2 & \le \E \|v(0)\|^2 +\|G\|^2_\infty\, \E \int_0^1 \|v(s)\|^2\, ds +
  2\|V\|_\infty\, \int_0^1 \|v(s)\|^2\, ds\\
& \qquad + 2 \E \sup_{0\le
                          t\le 1}
                          \left|\int_0^t(Gv,v)\, dW(s)\right|\\
& \le \E \|v(0)\|^2 + M_{G,V} \E \int_0^1\|v(s)\|^2\, ds + C
  \E\left(\int_0^1|(Gv,v)|^2\, ds\right)^{1/2}\\
&\le \E \|v(0)\|^2 + M_{G,V} \E \int_0^1\|v(s)\|^2\, ds + \frac{1}{2} \E
  \sup_{0\le t\le 1} \|v(t)\|^2\\
&\qquad  + C\|G\|^2_\infty\, \E\int_0^1\|v(s)\|^2\, ds.
\end{align*}
Therefore, 
\begin{align*}
 \E \sup_{0\le t\le 1} \|v(t)\|^2  & \le 2 \E \|v(0)\|^2 + 2(M_{G,V}+C\|G\|^2_\infty)
                                      \E\int_0^1\|v(t)\|^2\, dt\\
&\le 2 \E \|v(0)\|^2 + 2(M_{G,V}+C\|G\|^2_\infty) e^{M_{G,V}} \E\|v(0)\|^2\\
&\le e^{CM_{G,V}}\E \|v(0)\|^2.
\end{align*}
To justify the integration by parts and calculations carried out above, we use the same truncation and mollification as in \cite[Lemma 1]{Escauriaza2008}, which completes the proof.
\end{proof}

Interior regularity (or smoothing property) for deterministic parabolic equations is standard and well known, i.e., the solution becomes smooth for any $t>0$, even though the initial data may be singular. Similar but more subtle result for stochastic equations can be proved, see for example \cite{Flandoli1995}. However, in the rest of this section, we will show the interior regularity for the solution $u$ of stochastic equation \eqref{eq:main eqn}, with a quadratic exponential weight, which serves as an important tool for the rigorous justifications in the later sections. The result itself is also interesting and new in this stochastic context.

Let $\gamma \ge 0$. Fix $0<a<1$, let $\zeta_a$ be  a positive function in $C^\infty_c(\R)$ such that
\[
  \zeta_a(r) =
  \begin{cases}
    0,& r\le \max\{\gamma,2\}-1,\\
    2r^{-a}, & r\ge \max\{\gamma,2\}.
  \end{cases}
\]
We define $\varphi_a$ as a radial function in $\R^n$, i.e., $\varphi_a(x) = \varphi_a(|x|)$ satisfying
\begin{equation}
  \label{eq:varphi-a-1}
  \begin{cases}
    \varphi''_a(r) - \varphi'_a(r)/r = - a\zeta_a(r),\\
\varphi_a'(0) = 0 = \lim_{r\to\infty} \varphi_a''(r).
  \end{cases}
\end{equation}
Notice that $\varphi'_a(r) = ar \int_r^\infty \zeta_a(s)/s\, ds$. This allows us to choose $\varphi_a$ such that
\begin{equation}
  \label{eq:varphi-a-2}
  \varphi_a(r) =
\begin{cases}
  (2r^{2-a}-a)/(2-a), & \mbox{if}~r\ge \max\{\gamma,2\},\\
  (1 + O(a))r^2,& \mbox{if}~0\le r\le \max\{\gamma,2\}.
\end{cases}
\end{equation}

\begin{lemma}[Interior regularity]
  \label{lem:interiorReg}
Assume $u$ is a solution of equation \eqref{eq:main eqn}, and $V, G$ and $\nabla G$ are bounded in $[0,1]\times \R^n$. Then for any $\e>0$ we have
\begin{equation}
  \label{eq:interiorReg}
  \sup_{\e\le t\le 1}\E \|e^{\gamma \varphi_a(x)} \nabla u(t)\|^2 + \E \int_\e^1 \|e^{\gamma \varphi_a(x)} \cD^2 u(t)\|^2\, dt < \infty,
\end{equation}
where $\varphi_a$ is defined as in \eqref{eq:varphi-a-1} and \eqref{eq:varphi-a-2}.
\end{lemma}
\begin{proof}
Define $\psi(t,y) = \eta(t) \theta (y) \in C_0^\infty ([0,1]\times \R^n)$ with
\[
\eta(t) =
  \begin{cases}
    0,\quad t\in [0,\epsilon/4],\\
    1, \quad t\in [\e/2, 1],
  \end{cases}
\quad 
\theta(y) =
\begin{cases}
  1,\quad  y\in B_1(x),\\
0, \quad y\notin B_2(x),
\end{cases}
\]
for $x$ such that $|x|\ge N$, where $B_r(x)$ is a ball centered at $x$ with radius $r$.

Since $d(\psi u) = \partial_t \psi u\, dt + \psi\, du$, it follows from \Ito's formula applied to $\|\psi u\|^2$ that
\[
  d\|\psi u\|^2 = 2(\psi u, d(\psi u)) + \|G\psi u\|^2\, dt,
\]
or
\begin{multline*}
  \frac{1}{2} \|\psi(t) u(t)\|^2 = \int_0^t \int_{\R^n} \psi(s,y) \partial_s \psi(s,y) u^2(s,y)\, dy ds\\ + \int_0^t \int_{\R^n} \psi^2(s,y) u(s,y) \Delta u(s,y)\, dy ds + \int_0^1 \int_{\R^n} V(s,y) \psi^2(s,y)  u^2(s,y)\, dy ds\\
+ \int_0^t \int_{\R^n} G(s,y) \psi^2(s,y)  u^2(s,y)\, dy dW(s) + \frac{1}{2} \int_0^t \int_{\R^n} G^2(s,y) \psi^2(s,y)  u^2(s,y)\, dy ds,
\end{multline*}
where we have used the fact that 
$\psi(0) = 0$. Then by observing that
\begin{multline*}
  \int_{\R^n} \psi^2(t,y) u(t,y) \Delta u(t,y)\, dy =\\
 - \int_{\R^n} \psi^2(t,y) |\nabla u(t,y)|^2\, dy - 2 \int_{\R^2} \nabla \psi(t,y) \cdot \nabla u(t,y) \psi(t,y) u(t,y)\, dy,
\end{multline*}
and taking expectation we obtain
\begin{align*}
  & \frac{1}{2}\E \|\psi(t)u(t)\|^2 + \E \int_0^t \|\psi(s) \nabla u(s)\|^2\, ds \\
 =&\, \E \int_0^t\int_{\R^n} \psi(s,y)\partial_s \psi(s,y) u^2(s,y)\, dy ds
 - 2\E \int_0^t \int_{\R^n} \nabla \psi(s,y) \cdot \nabla u(s,y) \psi(s,y) u(s,y)\, dyds\\
& \qquad + \E \int_0^t \int_{\R^n} V(s,y)\psi^2(s,y) u^2(s,y)\, dyds 
+ \frac{1}{2} \E \int_0^t \| G(s)\psi(s) u(s)\|^2 \, ds.
\end{align*}
After using Cauchy-Schwarz inequality in the second integral on the right hand side of the previous equality, we have that there is a constant $C$ depending on 
$\|G\|_\infty$ and 
$\|V\|_\infty$ such that
\begin{equation}
  \label{eq:04042016}
  \sup_{\e\le t\le 1} \E \int_{y\in B_1(x)} u^2(t,y)\, dy + \E \int_{\e/2}^1 \int_{y\in B_1(x)} |\nabla u(t,y)|^2\, dy dt \le \frac{C}{\e}\E \int_{\e/4}^1 \int_{y\in B_2(x)} u^2(t,y)\, dydt.
\end{equation}

Next, let us differentiate the equation satisfied by $u$ with respect to a variable $x_i$ and we obtain
\[
  du_i = (\Delta u_i + (Vu)_i)\, dt + (Gu)_i\, dW(t),
\]
where $u_i = \partial_{x_i}u(t,x)$, and similarly for $(Vu)_i$ and $(Gu)_i$. Repeating the computations as before with $\psi(t,y) = \eta(t)\theta(y)$, where
\[
\eta(t) =
  \begin{cases}
    0,\quad t\in [0,\epsilon/2],\\
    1, \quad t\in [\e, 1],
  \end{cases}
\quad 
\theta(y) =
\begin{cases}
  1,\quad  y\in B_{1/2}(x),\\
0, \quad y\notin B_1(x),
\end{cases}
\]
we get
\begin{align*}
 &   \sup_{\e\le t\le 1}\E \int_{y\in B_{1/2}(x)} u_i^2(t,y)\, dy + \E \int_{\e}^1 \int_{y\in B_{1/2}(x)} |\nabla u_i(t,y)|^2\, dy dt\\
 \le &\, \frac{C}{\e}\E \int_{\e/2}^1 \int_{y\in B_1(x)}  u(t,y)^2\, dydt + C \E \int_{\e/2}^1 \int_{y\in B_1(x)} |\nabla u(t,y)|^2\, dydt.
\end{align*}
Thus, it follows from \eqref{eq:04042016} that
\begin{align}
  \label{eq:04042016-3}
 &   \sup_{\e\le t\le 1}\E \int_{y\in B_{1/2}(x)} |\nabla u(t,y)|^2\, dy + \E \int_{\e}^1 \int_{y\in B_{1/2}(x)} |\cD^2 u(t,y)|^2\, dy dt\notag\\
 \le &\, \frac{C}{\e}\E \int_{\e/4}^1 \int_{y\in B_2(x)}  u(t,y)^2\, dydt,
\end{align}
by summing in $i=1,2,\cdots,n$.

For $y\in B_2(x)$, and $|x|\ge N$ with $N$ sufficiently large, there is $\nu>0$ such that
\begin{equation*}
  (1-\nu)\varphi(y)\le (1-\nu/2) \varphi(x) \le \varphi(y),
\end{equation*}
where
\begin{equation}
  \label{eq:04042016-6}
  \varphi(x) = |x|^2/(1+4\gamma).
\end{equation}
 Therefore,
\begin{align}
  \label{eq:04042016-4}
  &   \sup_{\e\le t\le 1}\E \int_{y\in B_{1/2}(x)} e^{2(1-\nu)\gamma\varphi(y)}|\nabla u(t,y)|^2\, dy + \E \int_{\e}^1 \int_{y\in B_{1/2}(x)} e^{2(1-\nu)\gamma\varphi(y)}|\cD^2 u(t,y)|^2\, dy dt\notag\\
 \le &\,  \frac{C}{\e} 
\E \int_{\e/4}^1 e^{2(1-\nu/2)\gamma\varphi(x)}\int_{y\in B_2(x)} u^2(t,y)\, dydt\notag\\
\le & \,  \frac{C}{\e} 
\E \int_{\e/4}^1 \int_{y\in B_2(x)} e^{2\gamma\varphi(y)}u^2(t,y)\, dydt.
\end{align}
Now by means of a covering lemma, see for example \cite[Theorem 1.1]{Guzman1975}, we can find a sequence $\{x_j\}$ with $\sup_j|x_j|\ge N$ such that $\{|y|\ge N\} \subset \bigcup_j B_{1/2}(x_j)$ and $\sum_j \chi_{B_2(x_j)} \le C(n)$. Summing in $j$, we conclude from \eqref{eq:04042016-3} and \eqref{eq:04042016-4} that
\begin{align}
  &  \sup_{\e\le t\le 1}\E \int_{|y|\ge N} |\nabla u(t,y)|^2\, dy+ \E \int_{\e}^1 \int_{|y|\ge N} |\cD^2 u(t,y)|^2\, dy dt\notag\\
\le & \,  \frac{C}{\e} 
\E \int_{\e/4}^1 \int_{\R^n}u^2(t,y)\, dydt
< \infty,\label{eq:04042016-5}
\end{align}
and
\begin{align}
  &  \sup_{\e\le t\le 1}\E \int_{|y|\ge N} e^{2(1-\nu)\gamma\varphi(y)}|\nabla u(t,y)|^2\, dy+ \E \int_{\e}^1 \int_{|y|\ge N}e^{2(1-\nu)\gamma\varphi(y)} |\cD^2 u(t,y)|^2\, dy dt\notag\\
\le & \,  \frac{C}{\e} 
\E \int_{\e/4}^1 \int_{\R^n} e^{2\gamma\varphi(y)}u^2(t,y)\, dydt
< \infty,\label{eq:04042016-2}
\end{align}
by the energy estimate in Lemma \ref{lem:energy}.

Finally, we fix $N$ and $\nu$ such that \eqref{eq:04042016-5} and \eqref{eq:04042016-2} hold, and without loss of generality, we may assume $N\ge \max\{\gamma, 2\}$. In this case, we have
\begin{equation}
  \label{eq:04042016-7}
  \E \int_{|x|\ge N} e^{2\gamma\varphi_a(x)}|\nabla u(t)|^2\, dx
=  \E \int_{\{|x|\ge N\}\cap E_a} e^{2\gamma\varphi_a(x)}|\nabla u(t)|^2\, dx +  \E \int_{\{|x|\ge N\}\cap E_a^c} e^{2\gamma\varphi_a(x)}|\nabla u(t)|^2\, dx,
\end{equation}
where $E_a = \{x\in \R^n: 2|x|^{2-a}\le a + (2-a)(1-\nu)|x|^2/(1+4\gamma)\}$, and $E_a^c$ is the complement of $E_a$. 

For the first integral, we have that $\varphi_a(x)\le (1-\nu)\varphi(x)$, where $\varphi$ is defined in \eqref{eq:04042016-6}, and thus,
\[
  \E \int_{\{|x|\ge N\}\cap E_a} e^{2\gamma\varphi_a(x)}|\nabla u(t)|^2\, dx  \le \E \int_{|x|\ge N} e^{2(1-\nu)\gamma\varphi(x)}|\nabla u(t,x)|^2\, dx,
\]
which by \eqref{eq:04042016-2} implies that
\begin{equation}
  \label{eq:04042016-8}
  \sup_{\e \le t\le 1} \E \int_{\{|x|\ge N\}\cap E_a} e^{2\gamma\varphi_a(x)}|\nabla u(t)|^2\, dx< \infty.
\end{equation}

For the second integral, we use the fact that if $x\in E_a^c$, then $|x|< 2^{1/a}[(1+4\gamma)/(1-\nu)]^{1/a}$ when $a\in(0,1)$. In fact,
\begin{align*}
 &  2|x|^{2-a}> a + \frac{(2-a)(1-\nu)}{1+4\gamma}|x|^2\\
\Leftrightarrow &\, \frac{(2-a)(1-\nu)}{1+4\gamma}|x|^a < 2 - \frac{a}{|x|^{2-a}}<2\\
\Rightarrow & \, |x|^a< \frac{2(1+4\gamma)}{(2-a)(1-\nu)}<2\frac{1+4\gamma}{1-\nu}.
\end{align*}
Hence,
\begin{align*}
  \E \int_{\{|x|\ge N\}\cap E_a^c} e^{2\gamma\varphi_a(x)}|\nabla u(t,x)|^2\, dx
& \le \E \int_{\{|x|\ge N\}\cap \{|x|\le 2^{1/a}[(1+4\gamma)/(1-\nu)]^{1/a}\} } e^{2\gamma\varphi_a(x)}|\nabla u(t)|^2\, dx\\
& \le e^{2\gamma \varphi_a(2^{1/a}[(1+4\gamma)/(1-\nu)]^{1/a})} \E \int_{|x|\ge N} |\nabla u(t,x)|^2\, dx,
\end{align*}
which by \eqref{eq:04042016-5} implies that
\begin{equation}
  \label{eq:04042016-9}
  \sup_{\e \le t\le 1} \E \int_{\{|x|\ge N\}\cap E_a^c} e^{2\gamma\varphi_a(x)}|\nabla u(t,x)|^2\, dx< \infty.
\end{equation}

It follows from \eqref{eq:04042016-7}, \eqref{eq:04042016-8} and \eqref{eq:04042016-9} that
\[
  \sup_{\e\le t\le 1} \E \int_{|x|\ge N} e^{2\gamma\varphi_a(x)}|\nabla u(t,x)|^2\, dx< \infty.
\]
Similarly,
\[
  \E\int_\e^1  \int_{|x|\ge N} e^{2\gamma\varphi_a(x)}|\cD^2 u(t,x)|^2\, dx<\infty.
\]

For the space integral over the region $\{|x|<N\}$, all the estimates are obvious, and we finish the proof.
\end{proof}

\begin{remark}
  \label{rem:blowup}
  It is worthy to note that the estimates in Lemma \ref{lem:interiorReg} may blow up as $a$ or $\e$ goes to zero, which indicates that we cannot use the energy estimate directly to work with $\sup_{t\in[0,1]}\E \|e^{\gamma |x|^2} u(t)\|$. However, we are going to use this lemma qualitatively but not quantitatively, so this fact is not an issue for our purpose.
\end{remark}

\section{Logarithmic convexity}
We first introduce a formal calculation to be used frequently for the logarithmic convexity.
\begin{lemma}
  \label{lem:absLem}
Let $\cS$ and $\cA$ be a symmetric and a skew-symmetric operators, respectively, possibly dependent on the time variable. Suppose $V(t,x)$ and $G(t,x)$ are bounded functions in $[0,1]\times\R^n$, and a reasonable function $f(t,x)$ satisfies
\[
df = (\cS + \cA) f dt + V f dt + G f\, dW(t).
\]
We also assume that there exists a time dependent operator $\cS_t$ such that
\begin{equation}
  \label{eq:S_t}
  d(\cS f) = \cS_t f dt + \cS df.
\end{equation}
Then there is a function $Q(t)$ and a universal constant $N$ such that
\begin{gather}
  \frac{d^2}{dt^2} [\log H(t) + Q(t)] \ge \frac{2}{H(t)}\bigg\{\E (\cS_t f + [\cS, \cA]f, f)+  D_G(t)  -\frac{D(t)H_G(t)}{H(t)}\bigg\},\label{eq:d^2H}\\
\|Q(t)\|_\infty \le N (\|V\|_\infty + \|V\|^2_\infty + \|G\|^2_\infty),\label{eq:Q}
\end{gather}
where
\begin{equation}
  \label{eq:H-D}
  \begin{aligned}
    & H(t) = \E \|f\|^2,\  H_G(t) = \E \|Gf\|^2,\\
& D(t) = \E(\cS f, f), \ D_G(t) = \E (\cS(Gf), Gf).
  \end{aligned}
\end{equation}

\end{lemma}
\begin{proof}
 By \Ito's formula, we have
\begin{align*}
  d(f, f) & = 2(df, f) + \|Gf\|^2 dt\\
& = 2(\cS f, f) dt + 2(Vf, f) dt + 2(Gf, f) \, dW.
\end{align*}
Thus
\begin{equation}
  \label{eq:dH-1}
  \dot{H}(t) = 2D(t) + 2\E (Vf, f) + H_G(t).
\end{equation}
Let us rewrite $D$ as follows:
\begin{equation}
  \label{eq:D}
  D(t) = \frac{1}{2}\E (2\cS f+ Vf, f) -\frac{1}{2}\E (Vf, f),
\end{equation}
and so
\begin{equation}
  \label{eq:dH-2}
  \dot{H}(t) = \E (2\cS f + Vf, f) + \E(Vf, f) + H_G(t).
\end{equation}
Then
\begin{equation}
  \label{eq:DdH}
  D(t)\dot{H}(t) = \frac{1}{2}\left[|\E(2\cS f + Vf, f)|^2 - |\E(Vf, f)|^2\right] + D(t)H_G(t).
\end{equation}
By \Ito's formula again and equality \eqref{eq:S_t}, we have
\begin{align*}
  d(\cS f, f)
& = (d(\cS f), f) + (\cS f, df) + (\cS (Gf), Gf)dt\\
& = (\cS_t f,f) dt +  2(\cS f, df) + (\cS (Gf), Gf) dt\\
& = (\cS_t f,f) dt +  2 \|\cS f\|^2 dt + 2(\cS f, \cA f) dt + 2(\cS f, Vf) dt\\
&\quad  + 2(\cS f, Gf) \, dW + (\cS (Gf), Gf) dt\\
& = (\cS_t f, f) dt + ([\cS, \cA]f, f) dt +\frac{1}{2} \|2\cS f + Vf\|^2 dt -\frac{1}{2}  \|Vf\|^2 dt\\
&\quad  + 2(\cS f, Gf) \, dW + (\cS (Gf), Gf) dt,
\end{align*}
In the last equality, we use the identities
\[
2(\cS f, \cA f) = ([\cS, \cA]f, f),
\]
and
\[
2 \|\cS f\|^2 + 2(\cS f, Vf) = \frac{1}{2} \|2\cS f + Vf\|^2 -\frac{1}{2}  \|Vf\|^2.
\]
Thus
\begin{equation}
  \label{eq:dD}
  \dot{D}(t) = \E (\cS_t f+[\cS, \cA]f, f) + \frac{1}{2} \E \|2\cS f + Vf\|^2 -\frac{1}{2} \E \|Vf\|^2 + D_G(t).
\end{equation}
Thus by \eqref{eq:dH-1} we obtain
\[
\frac{d}{dt}[\log H(t)] = \frac{\dot{H}(t)}{H(t)} = \frac{2D(t)}{H(t)} - \dot{F},
\]
where $F$ verifies
\begin{equation}
  \label{eq:F}
  \dot{F}(t) = - \frac{2\E(Vf, f)}{H(t)} - \frac{H_G(t)}{H(t)},\quad F(0) = 0.
\end{equation}
Then it follows from \eqref{eq:DdH} and \eqref{eq:dD} that
\begin{align*}
  & \frac{d^2}{dt^2}[\log H(t) + F(t)]\\ 
   =\, & 2 \frac{d}{dt}\left[\frac{D(t)}{H(t)}\right] = \frac{2}{H(t)}\left[\dot{D}(t) - \frac{D(t)\dot{H}(t)}{H(t)}\right]\\
=\, & \frac{2}{H(t)}\bigg\{\E (\cS_tf+[\cS, \cA]f, f) + \frac{1}{2} \E \|2\cS f + Vf\|^2 -\frac{1}{2} \E \|Vf\|^2 + D_G(t)\\
& \qquad - \frac{1}{2H(t)}\left[|\E(2\cS f + Vf, f)|^2 - |\E(Vf, f)|^2\right] -\frac{D(t)H_G(t)}{H(t)}\bigg\}.
\end{align*}
By Cauchy-Schwarz inequality, we have
\[
|\E (2\cS f + Vf,f)|^2 \le \E \|2\cS f + Vf\|^2\, \E \|f\|^2,
\]
which, together with the inequalities
\[
|\E (Vf, f)|^2 \ge 0, \quad\text\quad \E \|Vf\|^2 \le \|V\|^2_\infty\, \E \|f\|^2,
\]
implies
\begin{multline*}
\frac{d^2}{dt^2}[\log H(t) + F(t)] \ge
  \frac{2}{H(t)}\bigg\{\E (\cS_tf+[\cS, \cA]f, f)+  D_G(t)  -\frac{D(t)H_G(t)}{H(t)}\bigg\}
 - \|V\|^2_\infty.
\end{multline*}
Now let $Q (t) = F(t) + t(1-t)\|V\|^2_\infty$, and then \eqref{eq:F} yields \eqref{eq:d^2H} and \eqref{eq:Q}, which completes the proof.

\end{proof}

Recall that according to Remark \ref{rem:blowup}, we are unable to control the term $\sup_{t\in [0,1]}\E \|e^{\gamma |x|^2}u(t)\|^2$ directly from the energy estimate. On the other hand, we cannot use this abstract result in the same way as in the deterministic case to control that quantity. However, the next lemma allows us to bound the corresponding $L^2$-norm in the time variable. Moreover, it also shows that both 
\[
  \E \|e^{\gamma |x|^2} |x| u(t)\|^2, \quad\text{and}\quad \E \|e^{\gamma|x|^2} \nabla u(t)\|^2
\]
are finite for almost every $t\in [0,1]$. The later fact will be used in the proof of the logarithmic convexity.

In the rest of this section, we will make another assumption on the potential $V$ and the noise $G$, due to the stochastic conformal transformation studied in Lemma \ref{lem:conTrans}.

\begin{assumption}
  \label{assump:V-G-2}
The measurable functions $V$ and $G: [0,1]\times\R^n \to \R$ satisfy the following conditions:
\begin{itemize}
  \item[(1)] $V$ is bounded on $[0,1]\times\R^n$, and $G\in C^1_b([0,1]\times \R^n);$
  \item[(2)] Given $\gamma>0$, there exists $\e>0$ such that $\ds \sup_{t\in [0,1]}|G(t,x)|\le \sqrt{\gamma}|x|^{-\e}$ for $|x|\ge \max\{\gamma,2\};$
  \item[(3)] $\displaystyle \lim_{L\to\infty} \sup_{t\in[0,1],|x|>L} |V(t,x)|= \lim_{L\to\infty} \sup_{t\in[0,1],|x|>L} |\nabla G(t,x)|=0$.
\end{itemize}
\end{assumption}

\begin{lemma}
  \label{lem:lem4}
Suppose $u$ is a solution of equation \eqref{eq:main eqn}, and Assumption \ref{assump:V-G-2} holds. Then, for $\gamma >  \|G\|_\infty^2/4$, there is $N = N(\gamma, \|V\|_\infty, \|G\|_\infty, \|\nabla G\|_\infty)$ such that
\begin{multline}
  \label{eq:log-int}
  \int_0^1 \E \|e^{\gamma |x|^2} u(t)\|^2\, dt + \int_0^1 t(1-t) \E \|e^{\gamma |x|^2} |x| u(t)\|^2\, dt + \int_0^1 t(1-t) \E \|e^{\gamma|x|^2} \nabla u(t)\|^2\, dt\notag\\
\le N\left(\E \|e^{\gamma |x|^2} u(0)\|^2 + \E \|e^{\gamma |x|^2}u(1)\|^2 + \sup_{t\in [0,1]} \E \|u(t)\|^2\right).
\end{multline}
\end{lemma}
\begin{proof}
Let $f = e^{\gamma \varphi}u$, where $\varphi = \varphi(x)$ is to be chosen. Then $f$ satisfies, formally
\[
df = (\cS f + \cA f + Vf)\, dt + G f \, dW(t),
\]
where 
\begin{equation}
  \label{eq:S-A}
  \cS = \Delta + \gamma^2 |\nabla \varphi|^2, \quad\text{and}\quad \cA = -2\gamma \nabla \varphi\cdot \nabla - \gamma \Delta\varphi
\end{equation}
are symmetric and skew-symmetric operators, respectively. We do calculations as in Lemma \ref{lem:absLem}, and recall that
\[
  \begin{aligned}
    \dot{H}(t)& = 2D(t) + 2\E (Vf, f) + H_G(t),\qquad\text{and}\\
 \dot{D}(t) & = \E ([\cS, \cA]f, f) + \frac{1}{2} \E \|2\cS f + Vf\|^2 -\frac{1}{2} \E \|Vf\|^2 + D_G(t),
  \end{aligned}
\]
where $H,H_G,D$ and $D_G$ are defined in \eqref{eq:H-D}.

Multiplying $\dot{H}(t)$ by $(1-2t)$ and integrating in $t\in [0,1]$, we get
\begin{align*}
  \int_0^1 (1-2t) \dot{H}(t)\, dt = 2\int_0^1 H(t)\, dt - H(0) - H(1).
\end{align*}
On the other hand,
\begin{align*}
 &  \int_0^1(1-2t) \dot{H}(t)\, dt\\
=&\, 2 \int_0^1 (1-2t) D(t)\, dt + 2\int_0^1 (1-2t) \E (Vf,f) \, dt + \int_0^1 (1-2t) H_G(t)\, dt\\
= &\, 2\int_0^1 D(t)\, d(t(1-t)) + 2\int_0^1 (1-2t) \E (Vf,f) \, dt + \int_0^1 (1-2t) H_G(t)\, dt\\
= &\, -2\int_0^1 t(1-t) \dot{D}(t)\, dt  + 2\int_0^1 (1-2t) \E (Vf,f) \, dt + \int_0^1 (1-2t) H_G(t)\, dt\\
= &\, -2\int_0^1 t(1-t) \E([\cS,\cA]f,f)\, dt - \int_0^1t(1-t) \E \|2\cS f + Vf\|^2\, dt\\
&\qquad + \int_0^1t(1-t) \E\|Vf\|^2 \, dt - 2\int_0^1 t(1-t) D_G(t)\, dt\\
& \qquad + 2\int_0^1 (1-2t) \E (Vf,f) \, dt + \int_0^1 (1-2t) H_G(t)\, dt.
\end{align*}
Therefore, we have 
\begin{multline*}
  H(0) + H(1) - 2\int_0^1 H(t)\, dt\\
\ge  \, 2\int_0^1 t(1-t) \E([\cS,\cA]f,f)\, dt
 - \int_0^1 t(1-t) \E \|Vf\|^2\, dt
-2 \int_0^1(1-2t) \E(Vf, f)\, dt\\ - \int_0^1 (1-2t) H_G(t)\, dt + 2\int_0^1 t(1-t) D_G(t)\, dt.
\end{multline*}

Now, formally, if $\varphi(x) = |x|^2$ we have from \eqref{eq:S-A} that
\[
  ([\cS,\cA]f,f) = 8\gamma \int_{\R^n} |\nabla f|^2\, dx + 32\gamma^3\int_{\R^n} |x|^2|f|^2\, dx,
\]
and
\begin{align*}
  (\cS(Gf),Gf) & = -\int_{\R^n} |\nabla(Gf)|^2\, dx + 4\gamma^2\int_{R^n} |x|^2|Gf|^2\, dx\\
& \ge -2 \int_{\R^n} |G|^2 |\nabla f|^2\, dx - 2\int_{\R^n} |\nabla G|^2 |f|\, dx.
\end{align*}
Thus,
\begin{align*}
  & 2\int_0^1 H(t)\, dt + 4(4\gamma-\|G\|^2_\infty)\int_0^1 t(1-t) \E \int_{\R^n}|\nabla f|^2\, dx dt\\
&\qquad  + 64\gamma^3\int_0^1 t(1-t) \E \int_{\R^n} |x|^2 |f|^2\, dx dt\\
\le &\, H(0) + H(1) +  \int_0^1 t(1-t) \E \|Vf\|^2\, dt
+ 2 \int_0^1(1-2t) \E(Vf, f)\, dt\\ 
&\qquad + \int_0^1 (1-2t) H_G(t)\, dt + 4 \int_0^1 t(1-t) \E \int_{R^n} |\nabla G|^2 |f|^2 \,dx  dt.
\end{align*}
By assumptions on $G$ and $V$, for a given $\e>0$, there exists $L>0$ such that when $|x|>L$ we have
\[
  \max\left\{\sup_{t\in[0,1]}|G(t,x)|, \sup_{t\in[0,1]}|V(t,x)|, \sup_{t\in[0,1]}|\nabla G(t,x)|\right\}\le \e.
\]
Therefore, we obtain that
\begin{align*}
  \int_0^1 t(1-t)\E \|Vf\|^2\, dt 
& = \int_0^1 t(1-t) \E\left[\int_{|x|\le L} |V|^2 |f|^2 \, dx + \int_{|x|> L} |V|^2 |f|^2\, dx\right] dt\\
& \le \frac{\|V\|^2_\infty}{4} e^{2\gamma L^2} \sup_{t\in[0,1]} \E \|u(t)\|^2 + \frac{\e^2}{4}\int_0^1 H(t)\, dt.
\end{align*}
In the same way, we also have
\begin{gather*}
  2\int_0^1 (1-2t) \E(Vf, f)\, dt \le 2\|V\|_\infty e^{2\gamma L^2} \sup_{t\in [0,1]}\E\|u(t)\|^2 + 2\e \int_0^1 H(t)\, dt,\\
\int_0^1 (1-2t) H_G(t)\, dt \le \|G\|^2_\infty e^{2\gamma L^2} \sup_{t\in [0,1]}\E\|u(t)\|^2 + \e^2 \int_0^1 H(t)\, dt,\\
4 \int_0^1 t(1-t) \E \int_{R^n} |\nabla G|^2 |f|^2 \,dx  dt \le \|\nabla G\|^2_\infty e^{2\gamma L^2} \sup_{t\in [0,1]}\E\|u(t)\|^2 + \e^2 \int_0^1 H(t)\, dt.
\end{gather*}
Putting everything together yields
\begin{align}
  & (2-9\e^2/4 -2\e) \int_0^1 H(t)\, dt + 4(4\gamma-\|G\|^2_\infty)\int_0^1 t(1-t) \E \int_{\R^n}|\nabla f|^2\, dx dt\notag\\
&\qquad  + 64\gamma^3\int_0^1 t(1-t) \E \int_{\R^n} |x|^2 |f|^2\, dx dt\notag\\
\le &\, H(0) + H(1) +  N\sup_{t\in [0,1]}\E \|u(t)\|^2.\label{eq:0325-1}
\end{align}
We can now choose $\e$ small enough so that $2-9\e^2/4 - 2\e>0$ and conclude the result by using the  inequality (2.21) in \cite{Escauriaza2008}
\begin{equation}
  \label{eq:(2.21)}
  2\E \int_{\R^n} |\nabla f|^2 + 4\gamma^2 |x|^2|f|^2\, dx \ge \E \int_{\R^n} e^{2\gamma |x|^2} |\nabla u|^2\, dx.
\end{equation}

In order to make the calculations above rigorous, we set $f_a = e^{\gamma \varphi_a}u$, where $\varphi_a$ satisfies \eqref{eq:varphi-a-1} and \eqref{eq:varphi-a-2}. Then
\[
  \partial_{ij} \varphi_a(x) = \frac{\varphi'_a(|x|)}{|x|} \delta_{ij} -a \frac{x_ix_j}{|x|^2}\zeta_a(|x|),
\]
and thus
\begin{equation}
  \label{eq:D^2varphi}
  \cD^2\varphi_a(x) =
  \begin{cases}
    2I_n + O(a)\sum_{i,j=1}^n E_{ij}, &\mbox{if}~0\le |x|\le \max\{\gamma,2\},\\
    2|x|^{-a} I_n + O(a)\sum_{i,j=1}^n E_{ij}, &\mbox{if}~|x|\ge \max\{\gamma,2\},
  \end{cases}
\end{equation}
where $I_n$ is an $n\times n$ identity matrix, and $E_{ij}$ is the elementary matrix whose only nonzero entry is a $1$ in $i$-th row and $j$-th column. Also, in this case we have
\begin{equation}
  \label{eq:Delta^2}
  \|\Delta^2 \varphi_a\|_\infty \le C(n) a.
\end{equation}
Thus, 
\begin{align}
  \int_{\R^n} \cD^2 \varphi_a\nabla f_a\cdot \nabla f_a\, dx
& = \int_{|x|\le \max\{\gamma,2\}} \cD^2 \varphi_a\nabla f_a\cdot \nabla f_a\, dx + \int_{|x|\ge \max\{\gamma,2\}}\cD^2 \varphi_a\nabla f_a\cdot \nabla f_a\, dx\notag\\
& \ge 2 \int_{|x|\le \max\{\gamma,2\}} |\nabla f_a|^2\, dx +  n \int_{|x|\le \max\{\gamma,2\}} O(a) |\nabla f_a|^2\, dx\notag\\
& \quad + 2 \int_{|x|\ge \max\{\gamma,2\}}|x|^{-a} |\nabla f_a|^2\, dx +  n \int_{|x|\ge \max\{\gamma,2\}} O(a) |\nabla f_a|^2\, dx,\label{eq:0327-1}
\end{align}
and so
\begin{align*}
 &  8 \gamma \int_{\R^n} \cD^2 \varphi_a\nabla f_a\cdot \nabla f_a\, dx - 4\int_{\R^n} |G|^2|\nabla f_a|^2\, dx\\
\ge &\, 4\int_{|x|\le \max\{\gamma,2\}}(4\gamma + nO(a) - \|G\|^2_\infty) |\nabla f_a|^2\, dx\\
&\quad + 4\int_{|x|\ge \max\{\gamma,2\}} (3\gamma + n O(a)) |x|^{-a}|\nabla f_a|^2\, dx + 4\int_{|x|\ge \max\{\gamma,2\}} (\gamma |x|^{-a}-|G|^2)|\nabla f_a|^2\, dx.
\end{align*}
By choosing $a$ small enough, the first two integrals on the right hand side of the above inequality are non-negative by the condition that $\gamma > \|G\|^2_\infty/4$, and so is the last one due to the decay of the noise $G$ in Assumption \ref{assump:V-G-2}.

Observing that
\begin{multline*}
([\cS,\cA]f_a,f_a) =4\gamma \int_{\R^n} \cD^2\varphi_a \nabla f_a\cdot \nabla f_a\, dx + 4\gamma^3 \int_{\R^n} \cD^2\varphi_a \nabla \varphi_a\cdot \nabla \varphi_a |f_a|^2\, dx\\ - \gamma \int_{\R^n} \Delta^2\varphi_a |f_a|^2\, dx,
\end{multline*}
and repeating the formal computations as before, we obtain
\begin{multline*}
   (2-9\e^2/4 -2\e-\gamma C(n)a) \int_0^1 H_a(t)\, dt
  + 4\gamma^3\int_0^1 t(1-t) \E  \int_{\R^n} \cD^2\varphi_a \nabla \varphi_a\cdot \nabla \varphi_a |f_a|^2\, dx dt\\
+ 4\int_{|x|\le \max\{\gamma,2\}}(4\gamma + nO(a) - \|G\|^2_\infty) |\nabla f_a|^2\, dx
+ 4\int_{|x|\ge \max\{\gamma,2\}} (3\gamma + n O(a)) |x|^{-a}|\nabla f_a|^2\, dx\\
 + 4\int_{|x|\ge \max\{\gamma,2\}} (\gamma|x|^{-a}-|G|^2)|\nabla f_a|^2\, dx\\
\le \, H_a(0) + H_a(1) +  N\sup_{t\in [0,1]}\E \|u(t)\|^2,
\end{multline*}
where $H_a(t) = \E \|f_a(t)\|^2$.
By letting $a$ tend to zero, we prove \eqref{eq:0325-1} rigorously.

Next, we would like to replace the term $\int_{\R^n} |\nabla f|^2\, dx$ by $\int_{\R^n} e^{2\gamma |x|^2} |\nabla u|^2\, dx$ in \eqref{eq:0325-1}. To do this, we notice that \eqref{eq:0325-1} holds for $f_\rho = e^{(\gamma-\rho)|x|^2}u$ as well. Then by using the same argument as the interior regularity result in Lemma \ref{lem:interiorReg}, we can justify \eqref{eq:(2.21)} with such $f_\rho$ for $t\in [\e, 1]$. In the end, we send $\rho$ and $\e$ to zero and complete the proof.
\end{proof}

Now we are ready to show the logarithmic convexity.

\begin{lemma}
  \label{lem:lem3}
  Suppose $u$ is a solution of equation \eqref{eq:main eqn}, and Assumption \ref{assump:V-G-2} holds. Let
\[
M := \|V\|_\infty, \quad M_0 := \|G\|_\infty, \quad M_1 := \|\nabla G\|_\infty, \quad\text{and}~\gamma > M_0^2/4.
\]
Assume $\E \|e^{\gamma|x|^2}u(0)\|, \E \|e^{\gamma|x|^2}u(1)\|, M, M_0$ and $ M_1$ are finite. Then for every $t\in (0,1)$, $\E \|e^{\gamma|x|^2}u(t)\|^2$ is ``logarithmically convex'' and there is a universal constant $N$ such that
\begin{equation}
  \label{eq:log-convex}
  \E \|e^{\gamma|x|^2}u(t)\|^2 \le e^{N(M+M^2+M_0^2+M_1^2)}[\E \|e^{\gamma|x|^2}u(0)\|^2]^{1-t} [\E \|e^{\gamma|x|^2}u(t)\|^2]^t.
\end{equation}
\end{lemma}

\begin{proof}
Let $f = e^{\gamma \varphi}u$, where $\varphi = \varphi(x)$ is to be chosen. Then a direct calculation from \eqref{eq:S-A} shows that
\begin{equation}
  \label{eq:est_D}
  (\cS f, f) = -\int_{\R^n} |\nabla f|^2 dx + \gamma^2\int_{\R^n} |\nabla \varphi|^2 |f|^2 dx \le \gamma^2\int_{\R^n} |\nabla \varphi|^2 |f|^2 dx,
\end{equation}
\begin{align}
  \label{eq:est_Dg}
  (\cS (Gf), Gf) & = -\int_{\R^n} |\nabla (Gf) |^2 dx + \gamma^2\int_{\R^n} |\nabla \varphi|^2 |Gf|^2 dx\notag\\
& \ge -2 \int_{\R^n} |\nabla f|^2 |G|^2 dx - 2\int_{\R^n} |\nabla G|^2 |f|^2 dx,
\end{align}
and
\begin{align}
  ([\cS,\cA]f,f) 
& = 4\gamma \int_{\R^n} \cD^2 \varphi \nabla f \cdot \nabla f dx + 4\gamma^3 \int_{\R^n} \cD^2\varphi \nabla\varphi\cdot \nabla\varphi |f|^2 dx\notag\\
& \qquad - \gamma \int_{\R^n}\Delta^2 \varphi |f|^2 dx\notag\\
& \ge 4\gamma \int_{\R^n} \cD^2 \varphi \nabla f \cdot \nabla f dx + 4\gamma^3 \int_{\R^n} \cD^2\varphi \nabla\varphi\cdot \nabla\varphi |f|^2 dx\notag\\
& \qquad - \gamma \|\Delta^2 \varphi\|_\infty  \int_{\R^n} |f|^2 dx\label{eq:est_SA}.
\end{align}
Then by \eqref{eq:est_D}, \eqref{eq:est_Dg} and \eqref{eq:est_SA}, we obtain
\begin{align}
  &  \frac{2}{H(t)}\, \bigg\{\E (\cS_tf+[\cS, \cA]f, f)+  D_G(t)  -\frac{D(t)H_G(t)}{H(t)}\bigg\}\notag\\
 \ge&\, \frac{2}{H(t)}\, \E \bigg \{
4\gamma \int_{\R^n} \cD^2 \varphi \nabla f \cdot \nabla f dx + 4\gamma^3 \int_{\R^n} \cD^2\varphi \nabla\varphi\cdot \nabla\varphi |f|^2 dx - \gamma \|\Delta^2 \varphi\|_\infty  \int_{\R^n} |f|^2 dx\notag\\
 &\qquad -2 \int_{\R^n} |\nabla f|^2 |G|^2 dx - 2\int_{\R^n} |\nabla G|^2 |f|^2 dx - \gamma^2 \|G\|^2_\infty \int_{\R^n} |\nabla \varphi|^2 |f|^2 dx\bigg\}\notag\\
 \ge&\, \frac{2}{H(t)}\, \E I(t)
-2\gamma \|\Delta^2\varphi\|_\infty -4 \|\nabla G\|^2_\infty,\label{eq:d^2H-f}
\end{align}
where $H, H_G, D$ and $D_G$ are defined in \eqref{eq:H-D}, and
\begin{equation}
  \label{eq:I}
  \begin{aligned}
  I(t)& =: 4\gamma \int_{\R^n} \mathcal{D}^2 \varphi \nabla f \cdot \nabla f dx -2 \int_{\R^n} |\nabla f|^2 |G|^2 dx 
\\
& + 4\gamma^3 \int_{\R^n} \cD^2\varphi \nabla\varphi\cdot \nabla\varphi |f|^2 dx - \gamma^2 \|G\|^2_\infty \int_{\R^n} |\nabla \varphi|^2 |f|^2 dx.
  \end{aligned}
\end{equation}

Now let $\varphi(x) = |x|^2$. Then $\nabla \varphi = 2x$, $\Delta^2 \varphi =0$, and $ \cD^2\varphi = 2I_n$, where $I_n$ is an $n\times n$ identity matrix. So in this case, 
\begin{align*}
  I(t) & = 8 \gamma \int_{\R^n} |\nabla f|^2\, dx - 2 \int_{\R^n} |\nabla f|^2 |G|^2\, dx\\
& \qquad  + 32 \gamma^3 \int_{\R^n} |x|^2 |f|^2\, dx - 4\gamma^2 \|G\|^2_\infty \int_{\R^n} |x|^2 |f|^2\, dx\\
& \ge  2(4\gamma - \|G\|^2_\infty) \int_{\R^n} |\nabla f|^2\, dx + 4\gamma^2(8\gamma -\|G\|^2_\infty) \int_{\R^n} |x|^2 |f|^2\, dx\\
& \ge 0,
\end{align*}
since $\gamma> \|G\|^2_\infty/4$ by the assumption. If we at the moment assume the formal calculations in Lemma \ref{lem:absLem}, and the computations yielding $I(t) \ge 0$ are correct (in fact, later we only show that $\E I(t) \ge 0$ for almost every $t\in [0,1]$, which is sufficient for our purpose), it follows from \eqref{eq:d^2H} and \eqref{eq:d^2H-f} that
\begin{equation}
  \label{eq:d^2log}
  \frac{d^2}{dt^2}[\log H(t) + \wt{Q}(t)] \ge 0,
\end{equation}
with a universal constant $N$ and a function $\wt{Q}=\wt{Q}(t)$ satisfying
\[
  \|\wt{Q}\|_\infty \le N(M+M^2+M_0^2+M_1^2),
\]
and thus we have \eqref{eq:log-convex}.

Next, we do the justification for the calculations involved above. To this end, we use the same mollification as in Lemma \ref{lem:lem4}, and set $f_a = e^{\gamma \varphi_a}u$, where $\varphi_a$ satisfies \eqref{eq:varphi-a-1} and \eqref{eq:varphi-a-2}. Using \eqref{eq:D^2varphi} yields
\begin{align}
&  \int_{\R^n} \cD^2\varphi_a \nabla \varphi_a\cdot\nabla \varphi_a |f_a|^2\, dx\notag\\
= &\, \int_{|x|\le \max\{\gamma,2\}} \cD^2\varphi_a \nabla \varphi_a\cdot \nabla\varphi_a |f_a|^2\, dx + \int_{|x|\ge \max\{\gamma,2\}} \cD^2\varphi_a \nabla \varphi_a\cdot\nabla \varphi_a |f_a|^2\, dx\notag\\
\ge&\, 2 \int_{|x|\le \max\{\gamma,2\}} |\nabla \varphi_a|^2 |f_a|^2\, dx +  n \int_{|x|\le \max\{\gamma,2\}} O(a)|\nabla \varphi_a|^2|f_a|^2\, dx\notag\\
& \quad + 2 \int_{|x|\ge \max\{\gamma,2\}}|x|^{-a} |\nabla \varphi_a|^2|f_a|^2\, dx +  n \int_{|x|\ge \max\{\gamma,2\}} O(a) |\nabla \varphi_a|^2|f_a|^2\, dx\notag\\
\ge&\,  2 \int_{\R^n} |\nabla \varphi_a|^2|f_a|^2\, dx +  n \int_{\R^n} O(a)|\nabla \varphi_a|^2|f_a|^2\, dx\notag\\
& \quad + 2 \int_{|x|\ge \max\{\gamma,2\}} (|x|^{-a}-1)|\nabla \varphi_a|^2|f_a|^2\, dx.\label{eq:0327-2}
\end{align}
Then we obtain \eqref{eq:d^2H-f} from \eqref{eq:0327-1} and \eqref{eq:0327-2} with $I_a$ such that
\begin{align}
  \label{eq:I_a}
  I_a(t) 
& \ge 2\int_{\R^n} (4\gamma + n O(a) - \|G\|^2_\infty)|\nabla f_a|^2\, dx\notag\\
& \qquad + 8\gamma\int_{|x|\ge \max\{\gamma,2\}} (|x|^{-a}-1) |\nabla f_a|^2\, dx\notag\\
& \qquad + \gamma^2\int_{\R^n} (8\gamma + n O(a) - \|G\|^2_\infty)|\nabla \varphi_a|^2|f_a|^2\, dx\notag\\
& \qquad + 8\gamma^3 \int_{|x|\ge \max\{\gamma,2\}} (|x|^{-a}-1)|\nabla \varphi_a|^2|f_a|^2\, dx.
\end{align}
It follows from the interior regularity that the calculations leading to \eqref{eq:0327-1} and \eqref{eq:0327-2} are justified. In particular, the right hand side of \eqref{eq:I_a} is finite for $t\in [\e, 1]$. By the fact that $\gamma > \|G\|^2_\infty/4$, and for $a$ small enough, we obtain $\E I_a(t) \ge F_a(t)$, where
\begin{equation}
  \label{eq:F_a-2}
  F_a(t):= 8\gamma \E \int_{|x|\ge \max\{\gamma,2\}}(|x|^{-a}-1)(|\nabla f_a|^2 + \gamma^2 |\nabla \varphi_a|^2 |f_a|^2)\, dx.
\end{equation}
From Lemma \ref{lem:lem4}, we notice that $F_a(t)$ converges to zero as $a$ decreases to zero, for almost every $t\in [0,1]$. 
Therefore, using \eqref{eq:d^2H} in Lemma \ref{lem:absLem} and \eqref{eq:d^2H-f}, we arrive at
\[
  \frac{d^2}{dt^2}[\log H_a(t) + Q(t)]\ge O(a) - 2\gamma \|\Delta^2\varphi_a\|_\infty - 4\|\nabla G\|^2_\infty, \quad t\in [\e,1], ~\text{a.e.},
\]
which implies the logarithmic convexity for $H_a(t)$ for $t\in [\e, 1]$. By sending $a$ and $\e$ to zero, we complete the proof.
\end{proof}

\begin{remark}
  \label{rem:G(t)}
If $G= G(t)$ is a function in $t$, independent of the space variable $x$, the term $D_G(t) - D(t) H_G(t)/H(t)$ on the left hand side of \eqref{eq:d^2H-f},  will disappear. So we can deal with the commutator part as the deterministic case in \cite{Escauriaza2008}, and get rid of Assumption \ref{assump:V-G} as well as the restriction on $\gamma$ to arrive at the logarithmic convexity. Moreover, it is easy to see that the noise does not play a role in the calculations of the next section, so for space-independent noise the deterministic result still holds in the stochastic setting.
\end{remark}

\section{Proof of main result}

It is noted that in Lemma \ref{lem:lem4} and \ref{lem:lem3}, we require that the solution $u$ has the same quadratic exponential decay for $t_0=0$ and $t_1=1$, but Theorem \ref{th:main} assumes no decay for the initial data. In order to overcome this issue, we introduce the following conformal transformation, also known as {\it Appell transformation} for our stochastic equation. 
\begin{lemma}
  \label{lem:conTrans}
Assume $u(t,x)$ verifies
\[
  du = (\Delta u + V(t,x)u)\, dt + G(t,x)u\, dW(t),\quad (t,x)\in [0,1]\times \R^n.
\]
Let $\alpha,\beta>0$ and set
\begin{equation}
  \label{eq:y-u}
y(t,x) = [a(t)]^{\frac{n}{2}} u(b(t),a(t)x) e^{\frac{a(t)\kappa |x|^2}{4}},
\end{equation}
where
\[
  \kappa = \sqrt{\frac{\alpha}{\beta}}- \sqrt{\frac{\beta}{\alpha}}, \quad a(t) = \frac{\sqrt{\alpha\beta}}{\alpha(1-t)+\beta t}\quad\text{and}\quad b(t) = \sqrt{\frac{\beta}{\alpha}}a(t) t.
\]
then $y$ verifies 
\begin{equation}
  \label{eq:y}
  dy = (\Delta y + \widetilde{V}  y) dt + \wt{g}y\, dW(b(t)),
\end{equation}
with
\[
  \wt{V}(t,x) = [a(t)]^2 V(b(t),a(t)x),\quad\text{and}\quad \wt{g}(t,x) = G(b(t),a(t)x).
\]
Moreover, for any $\gamma\in\R$,
\begin{equation}
  \label{eq:y-u norm}
  \E \|e^{\gamma |\cdot|^2} y(t)\|^2 = \E \|e^{(\gamma a^2(1-s) + \kappa a(1-s)/4)|\cdot|^2} u(s)\|^2,
\end{equation}
where $s = b(t)$.

\end{lemma}
\begin{proof}
  By \Ito's formula,
\begin{equation*}
  dy = \frac{n}{2} a^{\frac{n}{2}-1} a' u e^{\frac{a\kappa |x|^2}{4}} dt + a^{\frac{n}{2}} e^{\frac{a\kappa |x|^2}{4}} du(b(t),a(t)x) + a^{\frac{n}{2}} a' \frac{\kappa |x|^2}{4} e^{\frac{a\kappa |x|^2}{4}}u\, dt,
\end{equation*}
where
\[
  du(b(t),a(t)x) = du(b(t),z)\big |_{z=a(t)x} + du(\tau,a(t),x) \big |_{\tau = b(t)}.
\]
On one hand, observing that
\begin{align*}
  u(b(t),z) 
& = u(0,z) + \int_0^{b(t)} [\Delta u+ Vu](s,z)\, ds + \int_0^{b(t)} Gu(s,z)\, dW(s)\\
& = u(0,z) + \int_0^t [\Delta u + Vu](b(s),z) b'(s)\, ds + \int_0^t Gu(b(s),z)\, dW(b(s)),
\end{align*}
where the second equality follows from the time change formula for Brownian motions, or more generally, local martingales, see for example \cite[Proposition 1.5, page 181]{Revuz1999},
and
\[
  du(\tau,a(t)x) = a'(t) \nabla u(\tau,a(t)x) \cdot x\, dt,
\]
we have
\begin{multline}
  \label{eq:ae-dy}
  a^{-\frac{n}{2}}e^{-\frac{a\kappa |x|^2}{4}} dy =  \frac{n}{2} a^{-1} a' u\, dt + \left[(\Delta u + Vu)b' + a'\nabla u\cdot x + a' \frac{\kappa |x|^2}{4} u\right]\,dt
 +  G u \, dW(b(t)).
\end{multline}
On the other hand, 
\begin{equation}
  \label{eq:yxx}
a^{-\frac{n}{2}}e^{-\frac{a\kappa |x|^2}{4}} \Delta y 
= a^{2} \Delta u + \kappa a^{2} \nabla u\cdot x + \frac{n\kappa}{2}a u+ \frac{\kappa^2}{4}a^2 |x|^2 u.
\end{equation}
Then it follows from \eqref{eq:ae-dy} and \eqref{eq:yxx} that
\begin{multline*}
 a^{-\frac{n}{2}}e^{-\frac{a\kappa |x|^2}{4}}(dy -\Delta y\, dt)\\
 = \big[(b'-a^2)\Delta u + (a'-\kappa a^2) \nabla\cdot x + \frac{n}{2}a^{-1}(a'-\kappa a^2) u + \frac{\kappa}{4}(a'-a) |x|^2 u \big]\, dt \\
+ b'Vu \, dt + G u\, dW(b(t)).
\end{multline*}
Using the identities $a' = \kappa a^2$ and $b' = a^2$, we obtain \eqref{eq:y}. Then relation \eqref{eq:y-u norm} is a consequence of the transform \eqref{eq:y-u} and the fact that $a(t)a(s) =1$.

\end{proof}

\begin{remark}
  In Lemma \ref{lem:conTrans}, if $u$ is $\cF_t$-adapted, then $y$ defined in \eqref{eq:y} is $\cF_{b(t)}$-adapted.
\end{remark}
 
It is known that the conformal transformation $y$ is equivalent in probability law to the process $\wt{u}$ satisfying
\begin{equation}
  \label{eq:u-tilde}
  d\wt{u} = (\Delta \wt{u} + \wt{V}(t,x)\wt{u}) dt + \wt{G}(t,x) \wt{u}\, d\wt{W}(t),
\end{equation}
where
\begin{equation}
  \label{eq:VG}
   \wt{V}(t,x) = [a(t)]^2 V(b(t),a(t)x),\quad \wt{G}(t,x) = G(b(t),a(t)x)\sqrt{b'(t)},
\end{equation}
and $\wt{W}$ is another Wiener process. Since the norms we consider are under the probability expectation, we may by a slight abuse of notation denote by $\wt{u}$ the conformal transformation of $u$. Then we have
\[
  \E \|e^{\gamma |x|^2}\wt{u}(0)\|^2 = \E \|u(0)\|^2, \quad \E \|e^{\gamma |x|^2}\wt{u}(1)\|^2 = \E \| e^{|x|^2/\delta^2} u(1)\|^2,
\]
by choosing $\alpha = 1, \beta = 1 + 4\gamma$ and $\gamma = 1/(2\delta)$ in Lemma \ref{lem:conTrans}. We also have 
\[
  \|\wt{V}\|_\infty\le (1+ 4\gamma) \|V\|_\infty,\quad \|\wt{G}\|_\infty \le \sqrt{1+4\gamma}\|G\|_\infty,\quad\text{and}\quad \|\nabla \wt{G}\|_\infty \le (1+4\gamma) \|\nabla G\|_\infty,
\]
by the identities $b' = a^2$ and $\|a\|_\infty = \sqrt{1+4\gamma}$.

Finally, we are ready to prove our main theorem.

\begin{proof}[Proof of Theorem \ref{th:main}]

Fix $R>0$. For $\gamma = 1/(2\delta)> 1/2$, we can find $\mu$ and $\e$ such that 
\begin{equation}
  \label{eq:mu-e}
  \frac{1}{2(1-\e)}< \mu < \gamma.
\end{equation}
 Define 
\[
  \varphi(t,x) = \mu|x+Rt(1-t)e_1|^2+\frac{R^2t(1-t)(1-2t)}{6}-\frac{R^2t(1-t)}{16\mu},
\]
and write $H_\mu(t)=\mathbb{E}\|f(t)\|^2$ with $f=e^{\varphi}\wt{u}$, where $\wt{u}$ is defined in \eqref{eq:u-tilde}.

Next, we show that $H_\mu(t)$ is a logarithmically convex function. It follows from \eqref{eq:VG} that $\wt{V}$ and $\wt{G}$ satisfy Assumption \ref{assump:V-G-2}, as long as $V$ and $G$ satisfy Assumption \ref{assump:V-G}. Therefore, we can apply previous lemmas and the interior regularity to show that the subsequent formal computations are correct.

 The equation satisfied by $f$ can be written as
\[
df=(\mathcal{S}f+\mathcal{A}f+\wt{V}f)dt+\wt{G}f\,d\wt{W}(t),
\]
where $\wt{V}, \wt{G}$ are defined in \eqref{eq:VG}, and
\begin{gather*}
\mathcal{S}=\Delta+4\mu^2|x+Rt(1-t)e_1|^2+2\mu R(1-2t)(x_1+Rt(1-t))\\+(t^2-t+\frac16)R^2-\frac{R^2(1-2t)}{16\mu},\\
\mathcal{A}=-4\mu(x+Rt(1-t)e_1)\cdot\nabla-2\mu n,
\end{gather*}
are symmetric and skew-symmetric operators, respectively. We also have
\begin{multline*}
\mathcal{S}_t+[\mathcal{S},\mathcal{A}]=-8\mu\Delta+32\mu^3|x+Rt(1-t)e_1|^2+2\mu R^2(1-2t)^2\\+4\mu R(4\mu(1-2t)-1)(x_1+Rt(1-t))+(2t-1)R^2+\frac{R^2}{8\mu},
\end{multline*}
where $\cS_t$ is an operator satisfying \eqref{eq:S_t}. Thus, we have that
\[\begin{aligned}
\mathbb{E}(\mathcal{S}_tf+[\mathcal{S},\mathcal{A}]f,f)
&=32\mu^3\mathbb{E}\int_{\R^n}\left|x+Rt(1-t)e_1+\frac{(4\mu(1-2t)-1)R}{16\mu^2}e_1\right|^2|f|^2\,dx\\
&\quad+8\mu\mathbb{E}\int_{\R^n}|\nabla f|^2\,dx.
\end{aligned}\]
While, on the other hand, for a general function $v$ we have
\[\begin{aligned}
(\mathcal{S}v,v)=&-\int_{\R^n}|\nabla v|^2\, dx +4\mu^2\int_{\R^n}|x+Rt(1-t)e_1|^2|v|^2\, dx+2\mu R(1-2t)\int_{\R^n}(x_1+Rt(1-t))|v|^2\, dx
\\&+R^2(t^2-t+\frac16)\int_{\R^n} |v|^2\, dx -\frac{R^2(1-2t)}{16\mu}\int_{\R^n} |v|^2\, dx\\
=&-\int_{\R^n} |\nabla v|^2\, dx+4\mu^2\int_{\R^n} \left|x+Rt(1-t)e_1+\frac{(4\mu(1-2t)-1)R}{16\mu^2}e_1\right|^2|v|^2\, dx
\\&+\frac{R}{2}\int_{\R^n} (x_1+Rt(1-t))|v|^2\, dx -R^2\left(\frac{1}{12}-\frac{1-2t}{16\mu}+\frac{1}{64\mu^2}\right)\int_{\R^n}|v|^2\, dx.
\end{aligned}\]
Therefore,
\[\begin{aligned}
& \mathbb{E}(\mathcal{S}_tf+[\mathcal{S},\mathcal{A}]f,f)H_\mu(t)+\mathbb{E}(\mathcal{S}(\wt{G}f),\wt{G}f)H_\mu(t) -\mathbb{E}(\mathcal{S}f,f)\mathbb{E}\|\wt{G}f\|^2
\\
=&\, 32\mu^3\mathbb{E}\int_{\R^n}\left|x+Rt(1-t)e_1+\frac{(4\mu(1-2t)-1)R}{16\mu^2}e_1\right|^2|f|^2\, dx H_\mu(t)
\\
&\quad + 4\mu^2\mathbb{E}\int_{\R^n} \left|x+Rt(1-t)e_1+\frac{(4\mu(1-2t)-1)R}{16\mu^2}e_1\right|^2|\wt{G}f|^2\, dx  H_\mu(t)
\\&\quad 
-4\mu^2\mathbb{E}\int_{\R^n} \left|x+Rt(1-t)e_1+\frac{(4\mu(1-2t)-1)R}{16\mu^2}e_1\right|^2|f|^2\, dx \mathbb{E}\int_{\R^n}|\wt{G}f|^2\, dx
\\&
\quad+\left(\frac{R}{2}\mathbb{E}\int_{\R^n}(x_1+Rt(1-t))|\wt{G}f|^2\, dx -R^2\left(\frac{1}{12}-\frac{1-2t}{16\mu}+\frac{1}{64\mu^2}\right)\mathbb{E}\int_{\R^n}|\wt{G}f|^2\, dx\right)H_\mu(t)
\\&
\quad-\left(\frac{R}{2}\mathbb{E}\int_{\R^n}(x_1+Rt(1-t))|f|^2\, dx-R^2\left(\frac{1}{12}-\frac{1-2t}{16\mu}+\frac{1}{64\mu^2}\right)\mathbb{E}\int_{\R^n}|f|^2\, dx\right)\mathbb{E}\int_{\R^n}|\wt{G}f|^2\, dx
\\&
\quad +8\mu\mathbb{E}\int_{\R^n}|\nabla f|^2\, dx H_\mu(t)-\mathbb{E}\int_{\R^n}|\nabla(\wt{G}f)|^2\, dx H_\mu(t)+\mathbb{E}\int_{\R^n}|\nabla f|^2\, dx \mathbb{E}\int_{\R^n}|\wt{G}f|^2\, dx\\
=&\,32\mu^3\mathbb{E}\int_{\R^n}\left|x+Rt(1-t)e_1+\frac{(4\mu(1-2t)-1)R}{16\mu^2}e_1\right|^2|f|^2\, dx H_\mu(t)
\\
&\quad + 4\mu^2\mathbb{E}\int_{\R^n} \left|x+Rt(1-t)e_1+\frac{(4\mu(1-2t)-1)R}{16\mu^2}e_1\right|^2|\wt{G}f|^2\, dx  H_\mu(t)
\\&\quad 
-4\mu^2\mathbb{E}\int_{\R^n} \left|x+Rt(1-t)e_1+\frac{(4\mu(1-2t)-1)R}{16\mu^2}e_1\right|^2|f|^2\, dx \mathbb{E}\int_{\R^n}|\wt{G}f|^2\, dx
\\&
\quad+\frac{R}{2}\mathbb{E}\int_{\R^n}x_1|\wt{G}f|^2\, dx H_\mu(t) -\frac{R}{2}\mathbb{E}\int_{\R^n}x_1|f|^2\, dx \mathbb{E}\int_{\R^n}|\wt{G}f|^2\, dx
\\&
\quad +8\mu\mathbb{E}\int_{\R^n}|\nabla f|^2\, dx H_\mu(t)-\mathbb{E}\int_{\R^n}|\nabla(\wt{G}f)|^2\, dx H_\mu(t)+\mathbb{E}\int_{\R^n}|\nabla f|^2\, dx \mathbb{E}\int_{\R^n}|\wt{G}f|^2\, dx.
\end{aligned}\]
Then we can easily estimate the latter to get
\[\begin{aligned}
& \mathbb{E}(\mathcal{S}_tf+[\mathcal{S},\mathcal{A}]f,f) H_\mu(t) +\mathbb{E}(\mathcal{S}(\wt{G}f),\wt{G}f)H_\mu(t)-\mathbb{E}(\mathcal{S}f,f)\mathbb{E}\|\wt{G}f\|^2
\\
\ge &\,  (32\mu^3-4\mu^2\|\wt{G}\|_\infty^2)\mathbb{E}\int_{\R^n}\left|x+Rt(1-t)e_1+\frac{(4\mu(1-2t)-1)R}{16\mu^2}e_1\right|^2|f|^2\, dx H_\mu(t)
\\&+4\mu^2\mathbb{E}\int_{\R^n} \left|x+Rt(1-t)e_1+\frac{(4\mu(1-2t)-1)R}{16\mu^2}e_1\right|^2|\wt{G}f|^2\, dx H_\mu(t)
\\&+(8\mu-2\|\wt{G}\|_\infty^2)\mathbb{E}\int_{\R^n}|\nabla f|^2\, dx H_\mu(t) -2\|\nabla \wt{G}\|_\infty^2 H^2_\mu(t)
\\&+\frac{R}{2}\mathbb{E}\int_{\R^n} x_1|\wt{G}f|^2\, dx H_\mu(t) -\frac{R\|\wt{G}\|_\infty^2}{2}\mathbb{E}\int_{\R^n} x_1|f|^2\, dx H_\mu(t).
\end{aligned}\]

We now need to control the last two integrals by some expression of the type $-C H_\mu^2(t)$. In order to do so, let us fix $\mu$ satisfying \eqref{eq:mu-e}  and define
\[
m_\mu = \sup_{t\in [0,1]}\left|t(1-t)+\frac{4\mu(1-2t)-1}{16\mu^2}\right|.
\]
By some elementary analysis, we obtain
\begin{equation}
  \label{eq:mmu}
  m_\mu = \begin{cases}
    \frac{1}{4},&\mbox{if}~\mu \ge \frac{1+\sqrt{2}}{2},\\
    \frac{4\mu+1}{16\mu^2}, &\mbox{if}~\frac{1}{2}<\mu< \frac{1+\sqrt{2}}{2}.
  \end{cases}
\end{equation}

Let us first focus on the following integrals
\[
I_1 := 4\mu^2\int_{\R^n} \left|x+Rt(1-t)e_1+\frac{(4\mu(1-2t)-1)R}{16\mu^2}e_1\right|^2|\wt{G}f|^2\, dx + \frac{R}{2}  \int_{\R^n} x_1|\wt{G}f|^2\, dx.
\]
We split these integrals into the regions $\{|x|>\alpha R\}$ and $\{|x|\le \alpha R\}$, where $\alpha>m_\mu$ is to be chosen later. In the first region,
\[
\left|x+Rt(1-t)e_1+\frac{(4\mu(1-2t)-1)R}{16\mu^2}e_1\right|\ge|x|-m_\mu R> \left(1-\frac{m_\mu}{\alpha}\right)|x|>(\alpha - m_\mu)R,
\] 
and thus,
\[\begin{aligned}
&4\mu^2\int_{|x|>\alpha R} \left|x+Rt(1-t)e_1+\frac{(4\mu(1-2t)-1)R}{16\mu^2}e_1\right|^2|\wt{G}f|^2\, dx +\frac{R}{2} \int_{|x|>\alpha R} x_1|\wt{G}f|^2\, dx.
\\
& \ge \left(4\mu^2\left(1-\frac{m_\mu}{\alpha}\right)(\alpha - m_\mu) -\frac{1}{2}\right)R\int_{|x|>\alpha R}|x||\wt{G}|^2\, dx\\
& \ge 0,
\end{aligned}\]
if $\alpha$ satisfies
\begin{equation}
  \label{eq:alpha-mu}
  4\mu^2(\alpha - m_\mu)^2\ge \frac{\alpha}{2}.
\end{equation}
On the other hand, if $|x|\le \alpha R$,
\begin{equation}\label{xger2}
\frac{R}{2}\int_{|x|\le \alpha R} x_1|\wt{G}f|^2\, dx \ge -\frac{\alpha R^2\|\wt{G}\|_\infty^2}{2}\int_{\R^n}|f|^2\, dx.
\end{equation}
So we have
\[
  I_1 \ge -\frac{\alpha R^2\|\wt{G}\|_\infty^2}{2}\int_{\R^n}|f|^2\, dx,
\]
as long as the inequality \eqref{eq:alpha-mu} holds.

We are going to use the same approach now with
\begin{multline*}
I_2 := (32\mu^3-4\mu^2\|\wt{G}\|_\infty^2)\int_{\R^n}\left|x+Rt(1-t)e_1+\frac{(4\mu(1-2t)-1)R}{16\mu^2}e_1\right|^2|f|^2\, dx\\-\frac{R\|\wt{G}\|_\infty^2}{2}\int_{\R^n} x_1|f|^2\, dx.
\end{multline*}
Studying both integrals in the region $\{|x|>\alpha R\}$ as before, we get that the difference is non-negative if
\[
(32\mu^3-4\mu^2\|\wt{G}\|_\infty^2)\left(1-\frac{m_\mu}{\alpha}\right)(\alpha - m_\mu)-\frac{\|\wt{G}\|_\infty^2}{2}\ge 0,
\]
or equivalently,
\begin{equation}
  \label{eq:G^2}
   \|\wt{G}\|_\infty^2 \le \frac{64\mu^3(\alpha - m_\mu)^2}{8\mu^2(\alpha-m_\mu)^2+\alpha}.
\end{equation}
But we notice that $\|\wt{G}\|^2_\infty <4 \mu$ implies the inequality \eqref{eq:G^2} by \eqref{eq:alpha-mu}. 
Furthermore, for $|x|\le \alpha R$ we have the same bound as in \eqref{xger2}. Thus, in this case, we again have
\[
  I_2 \ge -\frac{\alpha R^2\|\wt{G}\|_\infty^2}{2}\int_{\R^n}|f|^2\, dx,
\]
as long as $\|\wt{G}\|_\infty^2 < 4\mu$.

Therefore,
\[
  \mathbb{E}(\mathcal{S}_tf+[\mathcal{S},\mathcal{A}]f,f) H_\mu(t) +\mathbb{E}(\mathcal{S}(\wt{G}f),\wt{G}f)H_\mu(t)-\mathbb{E}(\mathcal{S}f,f)\mathbb{E}\|\wt{G}f\|^2 \ge -\alpha R^2\|\wt{G}\|_\infty^2-2\|\nabla \wt{G}\|_\infty^2.
\]
It follows from \eqref{eq:d^2H} in Lemma \ref{lem:absLem} that
\begin{equation}\label{eq2}
\frac{d^2}{dt^2}(\log H_\mu(t)+Q(t))\ge-2\alpha R^2\|\wt{G}\|_\infty^2-4\|\nabla \wt{G}\|_\infty^2,
\end{equation}
or, in other words
\[
\frac{d^2}{dt^2}\left(\log H_\mu(t)+Q(t)- \alpha R^2\|\wt{G}\|_\infty^2 t(1-t)-2\|\nabla \wt{G}\|_\infty^2t(1-t)\right)\ge0,
\]
where $Q$ is a function satisfying \eqref{eq:Q}. Thanks to this property, now we get the following ``logarithmic convexity'' for $H_\mu(t)$, i.e.,
\[
H_\mu(t)\le e^{N(\|\wt{V}\|_\infty+\|\wt{V}\|_\infty^2+\|\wt{G}\|_\infty^2+\|\nabla \wt{G}\|_\infty^2)+\frac{\alpha R^2\|\wt{G}\|_\infty^2}{4}}H_\mu(0)^{1-t}H_\mu(1)^t.
\]

Finally, let us prove the uniqueness. We notice that, if $t=1/2$,
\[\begin{aligned} H_\mu(1/2) 
& = \mathbb{E}\int_{\mathbb{R}^n}e^{2\mu|x+\frac{R}{4}e_1|^2-\frac{R^2}{32\mu}}|u(1/2)|^2dx\\
& \ge \mathbb{E}\int_{|x|\le \e R/4}e^{2\mu|x+\frac{R}{4}e_1|^2-\frac{R^2}{32\mu}}|u(1/2)|^2dx\\
&\ge e^{\frac{R^2(4(1-\e)^2\mu^2-1)}{32\mu}} \mathbb{E}\int_{|x|\le \e R/4}|u(1/2)|^2dx.
\end{aligned}\]
Hence, we conclude that
\[
\mathbb{E}\int_{|x|\le \e R/4}|u(1/2)|^2dx\le Ce^{\frac{\alpha R^2\|\wt{G}\|_\infty^2}{4}-\frac{R^2(4(1-\e)^2\mu^2-1)}{32\mu}}.
\]
Letting $\mu$ increase to $\gamma$ and then $R$ tend to infinity, one gets that $u\equiv0$ when
\[
\gamma >\frac{1}{2},\ \ \|\wt{G}\|_\infty^2 <\min\left\{4\gamma, \frac{4\gamma^2-1}{8\alpha \gamma}\right\}.
\]

In the end, let us find $\alpha$ (the smallest possible should be the best) and determine more precisely the upper bound for $\wt{G}$. In fact, it suffices to solve a quadratic equation from \eqref{eq:alpha-mu} and choose $\alpha$ to be the root greater than $m_\mu$, i.e.,
\[
  \alpha = \frac{1+16\mu^2 m_\mu + \sqrt{1+ 32 \mu^2 m_\mu}}{16\mu^2},
\]
or by \eqref{eq:mmu},
\[
  \alpha = 
\begin{cases}
  \frac{1}{4}+\frac{1+\sqrt{8\gamma^2+1}}{16\gamma^2}, &\mbox{if}~\gamma \ge \frac{1+\sqrt{2}}{2},\\   \frac{2\gamma+1}{8\gamma^2}+\frac{\sqrt{8\gamma+3}}{16\gamma^2}, &\mbox{if}~\frac{1}{2}<\gamma< \frac{1+\sqrt{2}}{2}.
\end{cases}
\]
It is easy to see that in either case we have $4\gamma > (4\gamma^2-1)/(8\alpha \gamma)$, and so we require that $\|\wt{G}\|^2< (4\gamma^2-1)/(8\alpha \gamma)$,
or equivalently $\gamma$ and $G$ satisfy \eqref{eq:gammaG}.
\end{proof}

\section{Acknowledgement}
\label{sec:acknowledgement}
This material is based upon work supported by the National Science Foundation under Grant No. DMS-1440140 while both authors were in residence at the Mathematical Sciences Research Institute in Berkeley, California, during Fall 2015 semester. The first author is partially supported by the projects MTM2011-24054, IT641-13. Both authors would like to thank L. Escauriaza and L. Vega for fruitful conversations. The authors would also like to thank C. Mueller and A. Debussche for discussions on stochastic conformal transformation, and thank S. Lototsky and F. Flandoli for discussions on interior regularity for stochastic PDEs.

\end{document}